\documentclass[12pt]{article}

\usepackage[margin=1in]{geometry}  
\usepackage{graphicx}              
\usepackage{amsmath}               
\usepackage{amssymb}               
\usepackage{amsfonts}              
\usepackage{amsthm}                
\usepackage{mathtools}
\usepackage{enumerate}
\usepackage{xparse}
\usepackage{hyperref}
\hypersetup{hidelinks}
\usepackage[utf8]{inputenc}
\usepackage{color}
\usepackage{authblk}
\usepackage{thm-restate}

\newtheorem{theorem}{Theorem}[section]
\newtheorem{lemma}[theorem]{Lemma}

\newtheorem{corollary}[theorem]{Corollary}

\newtheorem{definition}[theorem]{Definition}

\newcommand{\N}{\mathbb{N}}      

\DeclarePairedDelimiter\abs{\lvert}{\rvert}
\DeclarePairedDelimiter\norm{\lVert}{\rVert}
\DeclarePairedDelimiter\floor{\lfloor}{\rfloor}
\DeclarePairedDelimiter\ceil{\lceil}{\rceil}

\makeatletter
\let\oldnorm\norm
\def\norm{\@ifstar{\oldnorm}{\oldnorm*}}
\makeatother

\makeatletter
\let\oldfloor\floor
\def\floor{\@ifstar{\oldfloor}{\oldfloor*}}
\makeatother

\makeatletter
\let\oldceil\ceil
\def\ceil{\@ifstar{\oldceil}{\oldceil*}}
\makeatother

\NewDocumentCommand\G{gggg}{%
	\IfNoValueTF{#1}{\mathcal{G}}{%
	\IfNoValueTF{#2}{\mathcal{G}\left({#1}\right)}{%
    \IfNoValueTF{#3}{\mathcal{G}\left({#1}, {#2}\right)}{%
    \IfNoValueTF{#4}{\mathcal{G}\left({#1}, {#2}, {#3}\right)}{%
				\mathcal{G}\left({#1}, {#2}, {#3}, {#4}\right)}}}}%
}
\NewDocumentCommand\tildeG{gggg}{%
	\IfNoValueTF{#1}{\mathcal{\tilde G}}{%
	\IfNoValueTF{#2}{\mathcal{\tilde G}\left({#1}\right)}{%
    \IfNoValueTF{#3}{\mathcal{\tilde G}\left({#1}, {#2}\right)}{%
    \IfNoValueTF{#4}{\mathcal{\tilde G}\left({#1}, {#2}, {#3}\right)}{%
				\mathcal{\tilde G}\left({#1}, {#2}, {#3}, {#4}\right)}}}}%
}

\NewDocumentCommand\reg{gg}{%
	\IfNoValueTF{#1}{\reg{\epsilon}}{%
	\IfNoValueTF{#2}{\(\left({#1}\right)\)-regular}{\(\left({#1}, {#2}\right)\)-regular}}%
}

\NewDocumentCommand\lreg{gg}{%
	\IfNoValueTF{#1}{\lreg{\epsilon}}{%
	\IfNoValueTF{#2}{\(\left({#1}\right)\)-lower-regular}{\(\left({#1}, {#2}\right)\)-lower-regular}}%
}

\newcommand{\ccint}[2] { \left[ {#1},{#2} \right] }

\newcommand{\set}[1] { \left[ {#1} \right] }
\newcommand*{\eps}{\epsilon}

\usepackage{color}
\usepackage[usenames,dvipsnames]{xcolor}

\title{Local resilience for squares of almost spanning cycles in sparse random graphs}
\author[ ]{Andreas Noever\thanks{author was supported by grant no.\ 200021 143338  of the Swiss National Science Foundation.}}
\author[ ]{Angelika Steger}
\affil[ ]{\small Department of Computer Science}
\affil[ ]{ETH Z\"urich, 8092 Z\"urich, Switzerland }
\affil[ ]{\tt {\{anoever|steger\}@inf.ethz.ch}}
\begin{document}

\maketitle

\abstract{
In 1962, Pósa conjectured that a graph $G=(V, E)$ contains a square of a Hamiltonian cycle if
$\delta(G)\ge 2n/3$. Only more than thirty years later Komlós, Sárkőzy, and Szemerédi  proved this
conjecture using the so-called Blow-Up Lemma. Here we extend their result to a random graph setting.
We show that for every $\eps > 0$ and $p=n^{-1/2+\eps}$ a.a.s.\ every subgraph of $G_{n,p}$ with minimum degree at least $(2/3+\eps)np$ contains the square of a cycle on $(1-o(1))n$ vertices.
This is almost best possible in three ways: (1) for $p\ll n^{-1/2}$ the random graph will
not contain any square of a long cycle (2) one cannot hope
for a resilience version for the square of a {\em spanning} cycle (as deleting all edges in the
neighborhood of single vertex  destroys this property)  and (3) for $c<2/3$ a.a.s.\ $G_{n,p}$ contains a subgraph with minimum degree at least $cnp$ which does not contain the square of a path on $(1/3+c)n$ vertices.
}

\section{Introduction}
A classical result of Dirac~\cite{dirac52} states that any graph $G$ on $n\ge 3$ vertices with minimum
degree $\delta(G)\ge n/2$ contains a Hamilton cycle. This result is not difficult and a proof can be
found in most text books on graph theory, see e.g.\ \cite{Bollobas1998modern,west}. One also easily checks
that the constant $1/2$ is best possible:  the complete bipartite graph on $(n-1)/2$ and $(n+1)/2$
vertices (assuming $n$ odd) has minimum degree $n/2-1/2$ but does not contain a Hamilton cycle.

In 1962, Pósa conjectured that $G(V,E)$ contains a square of a Hamiltonian cycle if $\delta(G)\ge
2n/3$. A square of a  cycle $C$ is the cycle $C$ together with all edges between
vertices that have distance $2$ in $C$. Again it is not difficult to see that the constant $2/3$ is
best possible, just consider the complete tripartite graph on $(n-1)/3$, $(n-1)/3$ and $(n+2)/3$ vertices (assuming $3$ divides $n-1$). However proving that minimum degree $2n/3$ actually suffices
turned out to be a difficult problem. It required the development of powerful tools, most notably
Szemerédi’s Regularity Lemma \cite{Szemeredi75,Szemeredi78}  and the so-called Blow-Up Lemma \cite{BlowUp97}, before Pósa's conjecture
was proven, at least for all sufficiently large~$n$~\cite{KSS96}.
\begin{theorem}[Komlós, Sárközy, Szemerédi]\label{theorem:kss}
There exists a natural number $n_0$ such that if $G$ has order $n$ with $n\geq n_0$ and 
\[
\delta(G)\geq \frac{2}{3}n,
\]
then $G$ contains the square of a Hamiltonian cycle.
\end{theorem}

In modern terminology the above results can also be stated as resilience statements. For a monotone
increasing graph property ${\cal P}$ the (local) resilience of a graph $G=(V,E)$ with respect to
${\cal P}$ is the minimum $r \in \mathbb R$ such that by deleting at each vertex $v \in V$ at most
an $r$-fraction of the edges incident to $v$ one can obtain a graph that does not have property
${\cal P}$. Dirac’s theorem implies that the local resilience of the property 'containing a Hamilton
cycle‘ of the complete graph is $1/2$, while the proof of Posa’s conjecture implies that  the
property 'containing a square of a Hamilton cycle‘ of the complete graph is $1/3$.

A natural extension for resilience results is to consider instead of the complete graph the random
graph $G_{n,p}$ and ask for the resilience as a function of the edge probability~$p$. It is natural
to expect that there exists a threshold $p_0$ so that for $p\gg p_0$ the local resilience of
$G_{n,p}$ of a property ${\cal P}$ is w.h.p.\ equal to the local resilience of the complete graph,
while for $p\ll p_0$ the random graph $G_{n,p}$ w.h.p.\ does not satisfy the property ${\cal P}$  at
all. Indeed, such a result is known, up to polynomial factors, for the property 'contains a Hamilton
cycle'. The threshold for existence of a Hamilton cycle is $p=(\log n + \log \log n + \omega(n))/n$
\cite{KOMLOS198355,BOLLOBAS198397}, while Lee and Sudakov~\cite{LS12} showed that for $p \gg \log
n/n$ the local resilience is $1/2-o(1)$.

The aim of this paper is to study the local resilience for the property  'containing a square of a
Hamilton cycle‘.  As it turns out for this problem already the  threshold for existence  is a hard
problem that, till today, is not yet completely understood. As a square of a Hamilton cycle contains
many triangles, $p\ge c / \sqrt{n}$ is certainly necessary, and it is conceivable that this may be
also be the true answer. The best bound known is $p\ge  C \log^4 n/\sqrt{n}$ \cite{Nenadov16}.

For the resilience problem one is thus tempted to speculate that at least for $p\ge
\operatorname{poly}(\log n)/\sqrt{n}$, for an appropriate polylog-factor, we have that the
resilience of $G_{n,p}$ with respect to the property 'containing a square of a Hamilton cycle' is
$1/3-o(1)$. However, for this property it is easy to see that this is far too optimistic. By
deleting all edges in the neighborhood of a vertex $v$ we can ensure that $v$ cannot be part of any
square of a cycle. Thus for any $p=o(1)$ we have that the resilience for 'containing a square of a
Hamilton cycle‘ is $o(1)$. In order to obtain a non-trivial result we thus need to weaken the
required property. One easily checks that for constant resilience the best one can
hope for is to find a square of a cycle that covers all but $\Theta(1/p^2)$ vertices. Here we show
an approximate version of such a best possible result.

\begin{restatable}{theorem}{mainTheorem}
For every $\gamma,\nu>0$ and $p=n^{-\frac{1}{2}+\gamma/2}$ a.a.s.\ every subgraph of $G_{n,p}$ with minimum degree at least $(2/3+\nu)np$ contains a square of a cycle on at least $(1-\nu)n$ vertices.
\end{restatable}

Our result should be compared to a recent result of Peter Allen, Julia Böttcher, Julia Ehrenmüller and Anusch Taraz \cite{sparseBandwidth}. The authors prove a sparse version of the bandwidth theorem, which in particular implies that $G_{n,p}$ has local resilience for the property 'contains a square of a cycle on $n-C/p^2$ vertices‘ as long as $p\gg (\log n/n)^{1/4}$. Their proof technique as well as previous universality results hit a natural barrier around $p=n^{-1/\Delta}$ where $\Delta$ denotes the maximum degree of the embedded graph. Note that this density is required for any greedy / sequential type of embedding scheme, as for $p\ll n^{-1/\Delta}$ the typical neighbourhood of any $\Delta$ vertices is empty and one thus need to design more sophisticated look-ahead schemes. We achieve this by designing some pruning process  that identifies edges that satisfy some good expansion properties. In this way  we obtain
 the first nontrivial resilience result for almost spanning subgraphs that achieves, up to polylog factors, the optimal density.

A natural question is whether a similar result holds for higher powers of a cycle. Applying the aforementioned bandwidth theorem gives a bound of $p\gg (\log n/ n)^{1/2k}$ for the $k$-th power of an almost spanning cycle. Similarly to the case $k=2$ this does not match the obvious lower bound of $p\geq n^{-1/k}$. Our approach does not generalize easily to this setting. This is mostly due to our reliance on sparse regularity techniques which yield very strong statements about the distribution of the edges, but not on larger structures like triangles or larger complete graphs. 


\section{Proof}
The proof makes heavy use of the sparse regularity lemma (see \cite{Kohayakawa1997216}) and related
techniques. The definition of an \reg{\eps}{p} graph is briefly stated below. For a more in depth
introduction to the topic see for example \cite{GerkeSteger05}.

\begin{definition}
	A bipartite graph $B=\left(U\cup W,E\right)$ is called \reg{\epsilon}{p} if for all $U'\subseteq U$
	and $W'\subseteq W$ with $\abs{U'}\geq\eps\abs{U}$ and $\abs{W'}\geq\eps\abs{W}$,
	\begin{displaymath}
		\abs{\frac{\abs{E\left(U',W'\right)}}{\abs{U'}\abs{W'}}-\frac{\abs E}{\abs{U}\abs{W}}}\leq\epsilon p.
	\end{displaymath}
	We write \reg{\epsilon} in case $p$ equals the density $\abs E/\left(\abs{U}\abs{W}\right)$.

	$B$ is called \lreg{\eps}{p} if for all $U'\subseteq U$ and $W'\subseteq W$ with $\abs{U'}\geq\eps\abs{U}$
	and $\abs{W'}\geq\eps\abs{W}$,
\[\frac{\abs{E(U',W')}}{\abs{U'}\abs{W'}}\geq(1-\eps)p.\]
\end{definition}

The next two lemmas follow immediately from the above definition.

\begin{lemma}\label{l:degreeok}
Let $B=\left(U\cup W,E\right)$ be an \lreg{\epsilon}{p}  bipartite graph. Then for every $W'\subseteq W$ of size at least $\eps \abs{W}$ all but at most $\eps|U|$ vertices in $U$ have  at least $(1-\eps)|W'|p$ neighbours in $W'$. If $B$ is \reg{\eps} with density $p$, then all but at most $2\eps|U|$ vertices in $U$ have  at least $(1-\eps)|W'|p$ and at most $(1+\eps)|W'|p$ neighbours in $W'$. \qed
\end{lemma}

\begin{lemma}\label{l:subgraphok}
Let $B=\left(U\cup W,E\right)$ be an \reg{\epsilon}  bipartite graph for some $\eps< 1/3$. Then every subgraph $\left(U\cup W,E'\right)$ of $B$ such that
$|E\setminus E'| \le \eps^4|E|$ is \reg{2\epsilon}. \qed
\end{lemma}

The so-called sparse regularity lemma allows us to partition every subgraph of a random graph $G_{n,p}$ in \reg{\eps} parts. To make this more precise we need two more definitions.
\begin{definition}
We say that a partition $(V_i)_0^k$ of a set $V$ is $(\eps,k)$-equitable if $\abs{V_0}\leq \eps \abs{V}$ and $\abs{V_1}=\dots=\abs{V_k}$. We call a partition $(V_i)_0^k$ $(\eps,p)$-regular if at most $\eps\binom{k}{2}$ pairs $(V_i,V_j)$ with $1\leq i < j \leq k$ are not \reg{\eps}{p}.
\end{definition}
\begin{definition}A graph $G=(V,E)$ is called $\eta$-upper-uniform with density $p$ if for all disjoint vertex sets $U,W\subseteq V$ of size at least $\eta \abs{V}$, we have
\[
\abs{E(U,W)}\leq (1+\eta)p\abs{U}\abs{W}.
\]
\end{definition}

\begin{theorem}[Sparse Regularity Lemma, Theorem 2 in \cite{Kohayakawa1997216}]\label{sparseregularity}
For any given $\eps>0$ and $k_{\min}\geq 0$ there are constants $\eta>0$ and $k_{\max}\geq k_{\min}$ such that any $\eta$-upper-uniform graph with density $0<p\leq 1$ admits an \reg{\eps}{p} $(\eps,k)$-equitable partition of its vertex set with $k_{\min}\leq k \leq k_{\max}$.
\end{theorem}

It is well-known (and an easy consequence of Chernoff's inequality) that a random graph $G_{n,p}$ with $p\gg 1/n$ is a.a.s.\  $\eta$-upper-uniform. The Sparse Regularity Lemma can thus be applied a.a.s.\ to the graph $G_{n,p}$. For our proof we need the slightly stronger statement that a.a.s.\ every subgraph of $G_{n,p}$ that satisfies some  minimum degree condition contains a particularly nice regular partition. The proof follows routinely by standard arguments. We include it for convenience of the reader.

\begin{corollary}\label{l:standard}
For every $\mu,\nu,\eps>0$  and every positive integer $r_{\min}$ there exists $\alpha(\nu)$ and $r_{\max}(\nu,\eps, r_{\min})$ such that for $p\gg 1/n$
a.a.s.\ every spanning subgraph $\tilde G\subseteq G_{n,p}$ with minimum degree at least $(\mu+\nu)np$ contains a partition of the vertices $V=V_0\cup V_1\cup\dots\cup V_r$, where $r\in[r_{\min}, r_{\max}] $, such that $\abs{V_0}\leq \eps n$, $\abs{V_1}=\dots=\abs{V_r}$ and such that for every $i$ there exist  at least $\mu r$ indices $j\in\set{r}\setminus\{i\}$ such that $\tilde G[V_i, V_j]$ contains a spanning \reg{\eps} subgraph with  $\floor{\abs{V_i}\abs{V_j}\alpha p}$ edges.
\end{corollary}
\begin{proof}
We choose $\alpha$, $k_{\min}$ and  $\eps_0$ such that  inequality $(*)$ from below and the following inequalities are satisfied simultaneously:
$$
\sqrt{\eps_0} \leq \frac{\nu}{2}, \qquad 2\eps_0/\alpha < \eps, \qquad (1-\sqrt{\eps_0})k_{\min}\geq r_{\min},\qquad\eps_0 + 2\sqrt{\eps_0}  \leq \eps.
$$
Suppose that $\tilde G$ is a spanning subgraph of $G_{n,p}$ with $\delta(\tilde G)\geq
(\mu+\nu)np$. For every $\eta>0$ a.a.s.\ every subgraph of the random graph $G_{n,p}$ is $\eta$-upper-uniform with density $p$ and thus the sparse regularity lemma can be applied to every subgraph of $G_{n,p}$. The sparse regularity lemma (\autoref{sparseregularity})
gives us a constant $k_{\max}(\eps_0, k_{\min})$ such that we  find an \reg{\eps_0}{p} $(\eps_0,k)$-equitable partition
$(V_i)_{0}^{k}$ of $\tilde G$ for some $k\in [k_{\min},k_{\max}]$.

Denote with $\tilde n\in\ccint{(1-\eps_0)n/k}{n/k}$ the size of the partition classes. For $i\in\set{k}$ define
\[
	d_i\coloneqq \{j\in \set{k}\setminus \{i\}\mid \text{$V_i$, $V_j$ has density at least $\alpha p$ in $\tilde G$}\}.
\]
Note that a.a.s.\ in $G_{n,p}$ we have that $\abs{E(V_i,V_j)} \leq (1+\eps_0) \tilde n^2p$ for all $i,j\in\set{k}$ and, with room to spare, $\abs{E(V_i, V_0\cup V_i)}\leq
2(\eps_0n+\tilde n) \tilde n p$ for all $i\in\set{k}$. The minimum degree condition of $\tilde G$ thus implies that
for all $i\in\set{k}$
\[ 
	\tilde n\left(\mu+\nu\right)np
	\leq E(V_i,V_0\cup V_i) + \sum_{\mathclap{j\in\set{k}\setminus i}} E(V_i, V_j)
	\leq  2(\eps_0 n+\tilde n)\tilde n p +  k \alpha p \tilde n^2 + d_i (1+\eps_0) \tilde n^2 p
\]
and thus
\[
	d_i
	\geq \frac{\left(\mu+\nu\right)n - 2(\eps_0 n +\tilde n) - k\alpha \tilde n}{(1+\eps_0)\tilde n}
	\geq  \frac{\left(\mu+\nu\right) k - 2(\frac{\eps_0}{1-\eps_0}k + 1)   -  k\alpha}{(1+\eps_0)}
	\stackrel{(*)}\geq \left(\mu+\frac{\nu}{2}\right) k.
\]

Observe that the fact that at most $\eps_0\binom{k}2$ pairs $(V_i,V_j)$ are not \reg{\eps_0}{p} implies that there are at most $\sqrt{\eps_0} k$ indices in $\set{k}$ for which the set
\[
	\{j\in \set{k}\setminus \{i\}\mid \text{$(V_i, V_j)$ is not  \reg{\eps_0}{p} graph}\}
\]
has size at least $\sqrt{\eps_0}k$. Every \reg{\eps_0}{p} graph with density at least $\alpha p$ is \reg{\eps_0/\alpha} and thus 
contains a \reg{2\eps_0/\alpha} subgraph with $\floor{\tilde n^2 \alpha p}$ edges (see for example Lemma 4.3 in
\cite{GerkeSteger05}).
As $\sqrt{\eps_0} \leq \frac{\nu}{2}$, we can thus find
a subset $R\subseteq \set{k}$ of at least $(1-\sqrt{\eps_0})k$ indices such that for every
$i\in R$ the set
\[
	\{j\in R\setminus \{i\}\mid \text{$V_i$, $V_j$ contain an \reg{2\eps_0/\alpha} graph with  $\floor{\tilde n^2\alpha p}$ edges}\}
\]
 is of size at least $\mu k$. Wlog.\ we may assume that $R=\left\{1,\dots,r\right\}$, where $r\geq (1-\sqrt{\eps_0})k\geq r_{\min}$. By choice of $\eps_0$ we know that the cardinality of  $V_0\cup \bigcup_{i> r} V_i$ is at most $\eps_0 n + \sqrt{\eps_0}k \tilde n \leq \eps n$. Thus $V_0\cup\bigcup_{i>r} V_i,V_1,\dots,V_r$ is the desired partition and the corollary thus holds for $r_{\max}=k_{\max}$.
\end{proof}

With \autoref{l:standard} at hand, an alert reader will certainly be able to guess our proof strategy:  apply the Komlós, Sárközy, Szemerédi Theorem to the partition guaranteed by \autoref{l:standard}  in order to find a square of a cycle for this partition and then find a long square of a cycle within this structure. The next definitions provide the notations to make this idea precise.
First we define the notion of a regular blow-up of a graph.
\begin{definition}
Denote with $\G{F}{n}{p}{\eps}$ the class of graphs that consist of $\abs{V(F)}$ pairwise disjoint
vertex sets of size $n$. Each vertex set represents a vertex of $F$, and two vertex sets span an
\reg{\eps} graph with density $(1\pm\eps)p$ whenever the corresponding vertices are adjacent in $F$.
\end{definition}

With $P_k$ we denote a path of length~$k$, with vertex set $\{1,\ldots,k+1\}$ and edges $\{i,i+1\}$ for $1\le i\le k$.
The cycle $C_k$ is obtained from $P_k$ by identifying the vertices $1$ and $k+1$. The square of a path is denoted with $P_k^2$ and the
 square of a cycle with $C^2_k$. For a graph $F\in\left\{P^2_k,C^2_k\right\}$ and its blow-up
$G\in\G{F}{n}{p}{\eps}$ we denote the vertex sets of $G$ by $V_1,\dots,V_{\abs {F}}$ and assume that the vertex set $V_i$ represents the $i$-th vertex along the path (or cycle). We collect some additional properties of $\G{P^2_k}{n}{p}{\eps}$ in the following definition: \begin{definition}\label{def:gtilde}
Denote with $\tilde G(P^2_{k}, n, p, \eps)\subseteq \G{P^2_{k}}{n}{p}{\eps}$
the class of graphs in which for $i\in\{1,\dots,k-1\}$
\begin{list}{}{}
\item[$(i)$] every edge spanned by $V_i, V_{i+1}$ closes a triangle with at least
	$(1-\eps)np^2$ vertices in $V_{i+2}$, and
\item[$(ii)$] all but $\eps n$ vertices $v\in V_{i+1}$ have neighboorhoods  into $V_i$ and $V_{i+2}$ which induce an
	\lreg{\eps}{p} subgraph. 
\end{list}
\end{definition}

We also need two auxiliary lemmas related to the above definition. Their proofs are deferred to \autoref{sec:gtilde} and \autoref{sec:pathExpansion}.
The first lemma states that in the random graph the restrictions imposed by \autoref{def:gtilde} are easy to satisfy:
\begin{restatable}{lemma}{gtilde}\label{lemma:gtilde}
For every $\alpha,\eps,\gamma,k_{\max}>0$ there exists $\eps'>0$ such that for every $\eta>0$ a.a.s.\ in $G_{n,p}$
with $p=n^{-1/2+\gamma/2}$ every subgraph $G\subseteq G_{n,p}$, where $G\in\G{P^2_k}{n_0}{\alpha p}{\eps'}$, $n_0\geq \eta n$, $k\leq k_{\max}$
contains a spanning subgraph from $\tildeG{P^2_k}{n_0}{\alpha p}{\eps}$.
\end{restatable}

For $G\in\tildeG{P^2_k}{n}{p}{\eps}$ with vertex partitions $V_1,\dots,V_k$ we say that an edge $e\in E(V_1, V_2)$ \emph{expands} to a set of edges $E'\subseteq E(V_{i},V_{i+1})$ if there exists a square of a path of length $i$ in $G$ between $e$ and every edge of $E'$.
The next lemma asserts that in the random graph every edge in $E(V_1,V_2)$ expands to a majority of the edges in $E(V_ {k},V_{k+1})$, whenever $k$ is sufficiently large.
Property (1) of \autoref{def:gtilde} already guarantees that every edge spanned by $V_1, V_2$ is contained in many
squares of a path ending in $E(V_ {k},V_{k+1})$. But since these paths may overlap this alone does not imply that every edge expands to a large portion of $E(V_ {k},V_{k+1})$. 

\begin{restatable}{lemma}{pathExpansion}\label{pathExpansion}
Let $\eta,\alpha,\gamma>0$ and $k\geq \frac{3}{\gamma}$ be fixed and let $p=n^{-1/2+\gamma/2}$. For
$\eps$ small enough depending on $\alpha$ a.a.s.\ every subgraph $G\subseteq G_{n,p}$ which is from
$G\in\tildeG{P^2_{k+4}}{n_0}{\alpha p}{\eps}$ for some  $n_0\geq \eta n$ satisfies the following:
all edges in $G[V_1\cup V_2]$ expand to at least a $0.52$-fraction of the edges in $G[V_{k+4}\cup V_{k+5}]$.
\end{restatable}

With these two lemmas at hand we can prove our main result.
\mainTheorem*
\begin{proof}
Suppose that $\tilde G$ is a spanning subgraph of $G_{n,p}$ with $\delta(\tilde G)\geq(\frac{2}{3}+\nu)np$. We will the fix constants $\alpha(\nu), \eps(\nu,\alpha), \eps'(\nu, \alpha, \eps, \gamma)>0$ throughout the proof.

Set  $k_0=\lceil\frac{3}{\gamma}\rceil+4$ and let $r_{\min}=\max\{3k_0, n_0\}$ where $n_0$ denotes the constant from \autoref{theorem:kss}.
For $\alpha(\nu)$ small enough we may invoke \autoref{l:standard} with $\mu\leftarrow 2/3$, $\nu\leftarrow \nu$,  $\eps\leftarrow\eps'$ to obtain, for some $r\in[r_{\min}, r_{\max}(\nu,\eps')]$, a partition $V=V_0\cup\dots\cup V_r$ of the vertices of $\tilde G$ such that $\abs{V_1}=\dots=\abs{V_r}=\tilde n\in[(1-\eps')n/r,n/r]$ and for every $i\in\set{r}$ the number of indices $j\in\set{r}\setminus\{i\}$ such that $\tilde G[V_i, V_j]$ contains a spanning \reg{\eps'} subgraph with $\floor{\tilde n^2\alpha p}$ edges is at least $\frac{2}{3}r$.

Since $r\geq n_0$ \autoref{theorem:kss} tells us that then $\tilde G$ must contain a subgraph $G\in \G{C^2_r}{\tilde n}{\alpha p}{\eps'}$ on the partitions $V_1,\dots,V_r$. We shall assume wlog.\ that $V_i$ represents the $i$-th vertex of the cycle. Furthermore for ease of notation we identify $V_{r+i}$ with $V_{i}$.

For integers $i\in\set{r}$ and $t\in \ccint{k_0}{2k_0}$ consider a collection of
subsets $V'_i\subseteq V_i,\dots$, $V'_{i+t}\subseteq V_{i+t}$ each of size $n'\geq\eps \tilde n$. By definition of \reg{\eps'}ity the sets $V'_{i},\dots,V'_{i+t}$ induce a subgraph $G'\in\G{P^2_t}{n'}{\alpha p}{2\eps'/\eps}$ in $G$. 
By \autoref{lemma:gtilde} we may pick $\eps'$ small enough depending on $\alpha,\eps,2k_0$ such that a.a.s.\ $G'$ contains a spanning subgraph $G_0\subseteq G'$ with $G_0 \in \tildeG{P^2_t}{n'}{\alpha p}{\eps}$.

We call an edge $e\in E(V'_i,V'_{i+1})$ \emph{good} (w.r.t.\ $V'_i,\dots,V'_{i+t}$) if it expands to
at least a $0.51$-fraction of the edges in $E(V'_{i+t-1}, V'_{i+t})$ (through $V'_i,\dots,V'_{i+t}$ in $G_0$).
\autoref{pathExpansion} with $\eta\leftarrow \eps/2r_{\max}$,
$\alpha\leftarrow \alpha$, $\gamma\leftarrow \gamma$, $k\leftarrow t-4\geq k_0-4 \geq \frac{\gamma}{3}$ tells us that all edges in
$G_0[V'_i,V'_{i+1}]$ expand to at least a $0.52$-fraction of the edges in $G_0[V'_{i+t-1},V'_{i+t}]$.
Since $G_0\subseteq G$ and since the density of $G_0$ and $G$ differs by at most $2\eps\alpha p$
this implies that say a $0.99$-fraction of the edges in  $E(V'_i,V'_{i+1})$ are good.

We can now find a long square of a path as follows: Fix an edge $e\in E(V_1,V_2)$ which is good with respect to
$V_1,\dots,V_{k_0+1}$. $e$ expands to a $0.51$ fraction of the edges spanned by $V_{k_0}, V_{k_0+1}$ and a
$0.99$ fraction of the edges in $E(V_{k_0}, V_{k_0+1})$ are good with respect to $V_{k_0},\dots,V_{2k_0}$. Thus
we may fix a square of a path $P\subseteq G[V_1\cup\dots\cup V_{k_0+1}]$ of length $k_0$ from $e$ to some
edge edge $e'\in E(V_{k_0},V_{k_0+1})$ which is good with respect to $V_{k_0},\dots,V_{2k_0}$.

Now remove $V(P)\setminus e'$ from $G$. Observe that since $r\geq 3k_0$  we did not remove any vertices from
$V_{k_0},\dots, V_{2k_0}$ and therefore $e'$ is still good (w.r.t.\ $V_{k_0},\dots, V_{2k_0}$). Thus we may extend
$P$ to end in an edge $e''\in E(V_{2k_0-1},V_{2k_0})$ which is good w.r.t.\ to $V_{2k_0-1},\dots,V_{3k_0-1}$. This
procedure can be continued for as long as $\abs{V_i}\geq \eps\tilde n$ for every $i\in\set{r}$. Thus we obtain a square of a path $P$ which uses at least $(1-2\eps)\tilde n$ vertices from each partition of $G$.

This construction can be generalized to obtain a square of the cycle: before fixing the first edge
$e$ set aside sets $\tilde V_1\subseteq V_1$, \dots, $\tilde V_{r}\subseteq V_r$ each of size $\eps\tilde n$.
Pick $e\in E(\tilde V_1, \tilde V_2)$ such that it expands (backwards) to at least a
$0.51$-fraction of the edges in $E(\tilde V_{r+1-k_0},\tilde V_{r+2-k_0})$. Then embed a long
path starting with $e$ in $V(G)\setminus \bigcup \tilde V_i$. At any point we may decide to close it
by picking the next edge such that it expands to a $0.51$-fraction of the edges in
$E(\tilde V_{r+1-k_0},\tilde V_{r+2-k_0})$.
With this construction we find a square of a cycle of length at least
\[
r\cdot(1-3\eps)\tilde n\geq (1-3\eps) (1-\eps')\frac{n}{r}\geq (1-\mu)n
\] provided
that $\eps', \eps$ are chosen small enough depending on $\mu$.
\end{proof}

\section{Proof of \autoref{lemma:gtilde}}\label{sec:gtilde}

For the proof of \autoref{lemma:gtilde} we need the following sparse
regularity inheritance theorem:
\begin{theorem}[Gerke, Kohayakawa, R\"{o}dl, and Steger~\cite{GKRS07}]\label{inheritance}
For $0<\beta,\eps'<1$, there exists $\eps_0=\eps_0(\beta,\eps')>0$ and
$C=C(\eps')$ such that, for any $0<\eps\leq\eps_0$ and $0<p<1$ every
\lreg{\eps}{p} graph $G=(V_1\cup V_2,E)$ satisfies that, for every $q_1,q_2\geq
Cp^{-1}$, the number of pairs of sets $(Q_1, Q_2)$ with $Q_i\subseteq V_i$ and
$\abs{Q_i}=q_i$ ($i=1,2$) that induce an \lreg{\eps'}{p} graph is at least
\[\left(1-\beta^{\min\left\{q_1,q_2\right\}}\right)\binom{\abs{V_1}}{q_1}\binom{\abs{V_2}}{q_2}.\]
\end{theorem}

From \autoref{inheritance} we easily deduce a bound on the number of graphs in $\G{K_3}{n}{p}{\eps_0}$ for which there exist `many' vertices in $V_1$ whose neighborhood does not induce a lower regular graph of `roughly' the expected size. 
\begin{lemma}\label{lregnb}
For every $\beta, \eps>0$ there exists $\eps_0>0$ and $C>0$ such that for all integers $n,m\in\N$  that satisfy
$3^nn^{2n}\le 2^{n^{3/2}}$ and  $m\geq C n^{3/2}$
the following holds. The number of graphs $G=(V_1\cup V_2\cup V_3, E)$ in $\G{K_3}{n}{m/n^2}{\eps_0}$ with $m$ edges between each two partitions for which more than $\eps n$ vertices in $v\in V_1$ have neighborhoods in $V_2$, $V_3$ which are not of size $(1\pm \eps)m/n$ or which do not induce an \lreg{\eps}{m/n^2} subgraph in $G[V_2,V_3]$ is at most
\[
\beta^{m}\binom{n^2}{m}^3.
\]
\end{lemma}

\begin{proof}
Write $p=m/n^2$ and define $\beta_0$ by $\beta_0^{(1-\eps)\eps/2}=\beta/2$.
Let $\eps_0=\min\left\{\eps/4,\eps_0(\beta_0,\eps)\right\}$ and $C=\max\left\{1,C(\eps)/(1-\eps)\right\}$, where $\eps_0(\cdot,\cdot)$
and $C(\cdot)$ are the functions given by \autoref{inheritance}.

Consider a graph $G=(V_1\cup V_2\cup V_3,E)$ for which the statement fails. For
any such graph we may partition $V_1$ into three sets $V_D, V_B, V_G\subseteq
V_1$ as follows: the set $V_D$ contains all vertices whose degree into at least one of
$V_2$ or $V_3$ is not within $(1\pm\eps)np$. $V_B$ contains all vertices whose
degree into both $V_2$ and $V_3$ is within $(1\pm\eps)np$ but whose
neighborhoods in $V_2$ and $V_3$ do not induce an \lreg{\eps}{p} graph in $G[V_2,V_3]$.
Finally set $V_G=V_1\setminus \left(V_D\cup V_B\right)$. \autoref{l:degreeok} implies that 
$\abs{V_D}\leq 2\eps_0 n \leq \eps n/2$. Therefore it suffices to enumerate graphs $G$ with
$\abs{V_{B}}\geq \eps n/2$.

We now construct all graphs $G$ which produce a partition with
$\abs{V_B}\geq \eps n/2$. First we pick the \reg{\eps_0} graph spanned by $V_2$, $V_3$.
There are at most $\binom{n^2}{m}$ choices.

Second we pick a partition $V=V_D\cup V_B \cup V_G$ and fix the degrees of all vertices $v\in V_1$ in $G[V_1,V_2]$ and $G[V_1,V_3]$. The number of choices is at most $3^n\cdot n^{2n}$. Finally we fix the actual
neighborhoods of the vertices of $V_1$. For a vertex $v\in V_D \cup V_G$ the
number of choices is at most
\[
\binom{n}{\deg_{G[V_1,V_2]}(v)}\binom{n}{\deg_{G[V_1,V_3]}(v)}.
\]
For $v\in V_B$ we have to select a neighborhood which does not induce an
\lreg{\eps}{p} subgraph in $G[V_2,V_3]$. Since $\deg_{G[V_1,V_i]}(v)\geq (1-\eps)np
\geq C(\eps)p^{-1}$ \autoref{inheritance} tells us that the number of such
neighborhoods is at most
\[
\beta_0^{(1-\eps)np}\binom{n}{\deg_{G[V_1,V_2]}(v)}\binom{n}{\deg_{G[V_1,V_3]}(v)}.
\]
Since $\abs{V_B} \geq \eps n/2$ the total number of choices for the
neighborhoods of the vertices in $V_1$ is bounded by
\[
\beta_0^{(1-\eps)np\abs{V_B}}\prod_{v\in V_1}\binom{n}{\deg_{G[V_1,V_2]}(v)}\binom{n}{\deg_{G[V_1,V_3]}(v)}
\leq \beta_0^{(1-\eps)\eps m/2} \binom{n^2}{m}^2
= \left(\frac{\beta}{2}\right)^{m}\binom{n^2}{m}^2,
\]
where the inequality follows from Vandermonde's identity. 
Since $3^n\cdot n^{2n} \leq 2^{n^{3/2}} \leq  2^m$ by assumption on $n$ and $m$, this  completes the proof.
\end{proof}

\autoref{lregnb} immediately implies the following corollary about the number of triangles spanned by almost all edges:
\begin{corollary}\label{triangles}
For every $\beta, \delta>0$ there exists $\eps_0>0$ and $C>0$ such that for $m\geq C n^{3/2}$ and $n$ sufficiently large 
the following holds. The number of graphs $G=(V_1\cup V_2\cup V_3, E)$ in $\G{K_3}{n}{m/n^2}{\eps_0}$, with $m$ edges between each two partitions, for which more than $\delta m$ edges of $G[V_1,V_2]$ are contained in fewer than $(1-\delta)n(m/n^2)^2$ triangles is at most
\[
\beta^{m}\binom{n^2}{m}^3.
\]
\end{corollary}
\begin{proof}
Let $p=m/n^2$ and choose $\eps$ small enough for $(1-\eps)^3\geq (1-\delta)$ to hold. Denote with $V'_1\subseteq V_1$ the set of vertices $v\in V_1$ whose neighborhoods in $V_2, V_3$ are of size $(1\pm \eps)np$ and induce an \lreg{\eps}{p} subgraph in $G[V_1, V_2]$. From \autoref{l:degreeok} we deduce that every vertex in $V_1'$ is incident to at least $(1-\eps)^2np$ edges (with endpoint in $V_2$) which are each contained in at least $(1-\eps)^2np^2\geq(1-\delta)np^2$ triangles. Therefore, the total number of edges which are contained in  fewer than $(1-\delta)np^2$ triangles is at most $m-\abs{V'_1}(1-\eps)^2np$. By choice of $\delta$, the only possibility that the desired condition is not fulfilled is thus that  $\abs{V'_1}\leq (1-\eps)n$. \autoref{lregnb} handles exactly this case -- and thus concludes the proof if we choose $C$ and $\eps_0$ as in this lemma.
\end{proof}

With \autoref{triangles} at hand we are now ready to prove \autoref{lemma:gtilde}. 

\begin{proof}[Proof of \autoref{lemma:gtilde}]
Observe first that it suffices to consider a fixed integer $k\leq k_{\max}$, as the lemma then follows by choosing the minimum $\eps'$ for all $k\leq k_{\max}$ and a trivial union bound argument. 

We first consider property $(i)$ of \autoref{def:gtilde}. The key fact here is that we want the property to hold for {\em every} edge, while
\autoref{triangles} guarantees this only for {\em `almost all'} edges. It is thus obvious what we need to do: show that we can remove edges appropriately. To this end \autoref{l:subgraphok} will come in very handy, as it shows that if we choose $\eps'$ small enough then we can take away edges repeatedly, while still keeping some regularity properties. 

Define constants as follows: $\eps_0(\alpha, \eps)$ will be fixed at the end of the proof, but will be small enough to satisfy $(1-\eps)\leq (1-\eps_0)^3$. Set $\beta = (\alpha/(4e))^3$ and for $i\in\set{k-1}$ define $\delta_i = (\eps_{i-1}/4)^4/2$ and $\eps_i = \min\left\{\delta_i/4, \eps_{cor}\left(\beta,\delta_i\right)\right\}$, where $\eps_{cor}(\cdot,\cdot)$ denotes the function $\eps(\beta,\delta)$ defined in \autoref{triangles}. Finally define $m_i=\ceil{(1-\eps_i)n^2_0p_0}$ and $\eps'=\eps_{k-1}/4$. Observe that $\eps_0>\delta_1>\eps_1>\dots>\delta_{k-1}>\eps_{k-1}>\eps'$.

We now proceed as follows: for $i=k-1$ down to $i=1$ we remove edges from $E(V_{i},V_{i+1})$ if they are not contained in enough triangles with $V_{i+2}$ {\em with respect} to the edge set  that survived the removal process in the previous round. That is, we remove all edges in $E(V_i,V_{i+1})$ which are contained in fewer than $(1-\eps)n_0p_0^2$ triangles with $V_{i+2}$, only taking in account edges that are still present.

Assume first that  for all $1\le i < k$ we remove at most $2\delta_i m_i$ edges. Then
the resulting subgraph $\tilde G$ is, by construction, such that the graph satisfies property $(i)$ of \autoref{def:gtilde}. We claim that we also have that all pairs are \reg{\eps_0} with density at least $(1-\eps_0)p_0$. Note that this implies that $\tilde G\in  G(P^2_{k}, n_0, p_0, \eps_0)$.
By definition of $\tildeG{P^2_k}{n_0}{p_0}{\eps'}$ we know that (before removing any edges) the pair  $(V_{i},V_{i+1})$ is \reg{\eps'} (and thus \reg{\eps_0/2}) with density at least $(1-\eps')p_0$. If we remove at most $2\delta_i m_i \leq (\eps_0/2)^4\abs{E(V_i,V_{i+1)}}$ edges from $E(V_{i},V_{i+1})$, then \autoref{l:subgraphok} implies that the remaining graph is 
 \reg{\eps_0} with density at least $(1-\eps'-2\delta_i)p_0 \geq (1-3\delta_i)p_0\geq (1-\eps_0) p_0$.

So assume the above condition does not hold. Let  $i$ denote the largest $1\le i< k$ such that when processing $E(V_i,V_{i+1})$ we  have to remove more than $2\delta_i m_i$ edges. We claim that then $G[V_i\cup V_{i+1}\cup V_{i+2}]$ contains one of the subgraphs enumerated by \autoref{triangles}.

Indeed, let $G'_{i+1,i+2}\subseteq G[V_{i+1},V_{i+2}]$ denote the subgraph obtained after removing the edges which do not satisfy property (i).
If $i=k-1$ then, since we do not touch the last partition, we have $G'_{i+1,i+2}=G[V_{k},V_{k+1}]$ which is trivially \reg{\eps_i/2} with density at least $(1-\eps_i)p_0$.
If $i<k-1$ then by maximality of $i$ the graph $G'_{i+1,i+2}$ is obtained by removing at most $2\delta_{i+1} m_{i+1}\leq (\eps_i/4)^4\abs{E(G[V_{i+1}, V_{i+2}])}$ edges from the \reg{\eps_i/4} graph $G[V_{i+1},V_{i+2}]$.
Therefore by \autoref{l:subgraphok} $G'_{i+1,i+2}$ is \reg{\eps_i/2} with density at least $(1-\eps'-2\delta_{i+1})p_0\geq(1-\eps_i)p_0$.
Let $G_{i,i+1}\subseteq G[V_{i},V_{i+1}]$, $G_{i,i+2}\subseteq G[V_{i},V_{i+2}]$, $G_{i+1,i+2}\subseteq G'_{i+1,i+2}$ denote spanning \reg{\eps_i} subgraphs with exactly $m_i$ edges each (for $m_i\gg n_0$ and $\eps'\leq \eps_i/2$ such subgraphs always exists, see Lemma 4.3 in \cite{GerkeSteger05}). Observe that to obtain $G_{i,i+1}$ we removed at most $2\eps_i n_0^2p_0\leq \delta_i m_i$ edges from $G[V_{i},V_{i+1}]$.

Thus by choice of $i$ there have to be at least $\delta_i m_i$ more edges in $G_{i,i+1}$ which are each contained in fewer than $(1-\eps)n_0p_0^2$ triangles with $V_{i+2}$ in $G_i \coloneqq G_{i,i+1}\cup G_{i,i+2}\cup G_{i+1,i+2}$.
Since $(1-\eps)n_0p_0^2\leq(1-\eps_0)^3n_0p_0^2\leq (1-\delta_i)n_0(m_i/n_0^2)^2$ and $\eps_i\leq \eps_{cor}\left((\alpha/(4e))^3,\delta_i\right)$ we may apply \autoref{triangles} (with $\beta\leftarrow \beta = (\alpha/(4e))^3$, $\delta\leftarrow \delta_i$, $m\leftarrow m_i$) to conclude that $G_i$ must be one of at most
\[
\beta^{m_i}\binom{n_0^2}{m_i}^3,
\]
graphs enumerated by \autoref{triangles}.
The probability that $G_{n,p}$ contains any of these graphs as a subgraph is at most
\[
{n\choose n_0}^3\cdot \beta^{m_i}\binom{n_0^2}{m_i}^3\cdot p^{3m_i}
\leq 2^{3n} \beta^{m_i}\left(\frac{en_0^2p}{m_i}\right)^{3m_i}
\leq 2^{3n}\beta^{m_i} \left(\frac{2e}{a}\right)^{3m_i} =  2^{3n-3m_i}.
\]
Since $n_0\geq \eta n$ we have $m_i\gg n$ and may additionally union bound over all choices for $n_0$ and conclude that a.a.s.\ no such
graph appears in $G_{n,p}$. In particular our procedure never removes more than $2\delta_i m_i$ edges and produces a graph
$\tilde G\in\G{P^2_k}{n_0}{p_0}{\eps_0}$ which satisfies property $(i)$.

It remains to show that additionally $\tilde G$ also satisfies property (ii). Fix $\eps_0=\min\left\{\eps/3,\eps_{lem}\left(\beta,\eps/2\right)\right\}$, where $\eps_{lem}(\cdot,\cdot)$ denotes the function $\eps(\beta,\eps)$ defined in \autoref{lregnb}. This choice is valid since $(1-\eps)\leq (1-(\eps/3))^3$.
Suppose that $\tilde G$ fails property (ii) for some $i\in\set{k-1}$. We claim that then $\tilde G[V_{i}\cup V_{i+1}\cup V_{i+2}]$ must
contain a subgraph enumerated by \autoref{lregnb}. To this end let $\tilde G'\subseteq \tilde G[V_{i}\cup V_{i+1}\cup V_{i+2}]$ denote a spanning subgraph in which every partition
contains exactly $m_0=\ceil{(1-\eps_0)n^2_0p_0}$ edges and is \reg{2\eps_0} (as before see Lemma 4.3 in \cite{GerkeSteger05}).

For $\eps_0\leq \eps/2$ we have $(1-\eps/2)m_0/n_0^2\geq (1-\eps)p$ and thus every \lreg{\eps/2, m_0/n_0^2} graph is also \lreg{\eps, p_0}. It follows that  $\tilde G'$ must be among the at most 
\[
\beta^{m_0}\binom{n_0^2}{m_0}^3
\]
graphs enumerated by \autoref{lregnb} (with $\eps\leftarrow \eps/2$, $\beta\leftarrow\beta$ and $m\leftarrow m_0$). As before a union bound shows that a.a.s.\ $G_{n,p}$ does not contain such a graph and thus $\tilde G$ also satisfies property (ii).
\end{proof}

\section{Proof of \autoref{pathExpansion}}\label{sec:pathExpansion}
\autoref{def:gtilde} already implies that every edge spanned by $V_1, V_2$ is contained in many
copies of $P^2_{k+4}$. The goal of this section is to show that these paths also reach a majority of
the edges spanned by $V_{k+4}, V_{k+5}$. As a first step we show that two edges cannot span too many
copies of $P^2_k$.
\begin{lemma}\label{spannedPaths}
For every $\gamma>0$, $k>\frac{3}{2}\cdot \frac{1+\gamma}{\gamma}$ and $p=n^{-(1-\gamma)/2}$ a.a.s.\
no pair of edges in $G_{n,p}$ is connected by more than $2n^{k-3}p^{2k-3}$ squares of paths of
length $k$.
\end{lemma}
\begin{proof}
This follows directly from a theorem of Spencer on the number of graph extensions in the random
graph \cite{spencer90}. $P^2_k$ rooted at both of its end-edges is strictly rooted balanced and
$n^{k-3}p^{2(k-3)+3}\gg \log n$. Thus the number of squares of paths of length $k$ connecting any two edges is
concentrated around its expectation.
\end{proof}
The previous lemma easily implies a weaker version of \autoref{pathExpansion} where the constant
$0.52$ has to be replaced with a small value that depends on $\alpha$ and $\eta$. And this will indeed be the first step in the proof.
Next we prove a small lemma which contains two ad hoc arguments that will allow us to go from a small set of the edges
in $V_i,V_{i+1}$ to a slightly larger set of edges in $V_{i+1},V_{i+2}$.
\begin{lemma}\label{triangleExpansion}
Let $\eta,\alpha,\gamma,\eps>0$ be fixed. For $p=n^{-1/2+\gamma/2}$ in $G_{n,p}$ a.a.s.\ every copy
of $G=(V_1\cup V_2\cup V_3, E)\in\tildeG{K_3}{n_0}{p_0}{\eps}$, where $n_0\geq \eta n$ and
$p_0=\alpha p$, satisfies the following: for every set of edges $\tilde E\subseteq
E\left(V_1,V_2\right)$ denote with $\triangle (\tilde E)$ the number of edges in $E(V_2, V_3)$ which
form a triangle with some edge of $\tilde E$ in $G$. Let $\tilde V_2\subseteq V_2$ denote the
vertices which are incident to some edge of $\tilde E$ and set $s=\min_{v\in \tilde V_2} \deg_{\tilde E}(v)$. Then
\begin{equation*}
\triangle(\tilde E) \geq
\begin{cases}
\abs{\tilde V_2} n_0 p_0^2 /(2p) 					& \text{if } s\geq \log ^2(n)n_0p/n^\gamma,\\
(1-\eps)^2 \left(\abs{\tilde V_2}-5\eps n_0\right) n_0p_0 	& \text{if } s\geq 2\eps n_0p_0.
\end{cases}
\end{equation*}
\end{lemma}
\begin{proof}
Fix $v\in \tilde V_2$.  Write $V^v_1=\Gamma_{\tilde E}(v)$ and let
$V^v_3=V_3\cap\Gamma(v)\cap\Gamma(V^v_1)$ denote the subset of vertices of $V_3$ which form a
triangle with $v$ and some edge from $\tilde E$. By definition of $G$ every edge in $\tilde E$ is
contained in at least $(1-\eps)n_0p_0^2$ triangles with $V_3$ and therefore
\begin{equation}\label{eq:l42}
	E(V^v_1, V^v_3)
	\geq (1-\eps)\abs{V^v_1}n_0p^2_0.
\end{equation}

A.a.s.\ in $G_{n,p}$ all disjoints sets $A$, $B$ of sizes $\abs{A}\geq \log^2(n) n_0p_0/n^\gamma$
and $\abs{B}\leq b = n_0 p^2_0/(2p)$ have at
most $(1+\eps)\abs{A} b p$ edges between them. To see this, observe that $\abs{A}bp \gg \log n\max\left\{\abs{A},b\right\}$ and the claim thus follows from Chernoff's inequality together with a straightforward union bound argument over all sets of size $\abs A$ and at most $b$. 

We now consider the two cases. If $s\geq
\log^2(n)n_0p_0/n^\gamma$ then $\eqref{eq:l42}$ implies
$E(V^v_1, V^v_3) \geq (1-\eps) \abs{V^v_1} n_0 p_0^2 = 2(1-\eps) \abs{V^v_1}bp$.
For $\abs{V_3^v}\le b$ this would contradict the bounds from the Chernoff inequality in the previous paragraph; thus $\abs{V_3^v}\ge b$, implying the desired bound.


Now suppose that $s\geq 2\eps n_0p_0$ and that $v$ is such that its neighboorhoods in $V_1$ and $V_3$ are of size
$(1\pm\eps)n_0p_0$ and induce a \lreg{\eps}{p} subgraph. As $\abs {V^v_1} \geq s$ the assumption on $p$ and $v$ imply $\abs {V^v_1}  \geq \eps \abs{\Gamma(v)\cap V_1}$. Thus we can apply 
 \autoref{l:degreeok}  to deduce that at most $\eps\abs{\Gamma(v)\cap V_3}$ vertices in $\Gamma(v)\cap V_3$ have no neighboor in $V^v_1$. Thus 
$\abs{V^v_3}\geq (1-\eps)\abs{\Gamma(v)\cap V_3} \geq (1-\eps)^2n_0p_0$.
By \autoref{l:degreeok} at most $4\eps n_0$ vertices do have neighborhoods of the wrong size in either $V_1$ or $V_3$. By definition of $\tilde{\mathcal G}$ at most $\eps n_0$ vertices have neighborhoods which do not induce an \lreg{\eps}{p} subgraph. Thus $\triangle(\tilde E)\geq
\left(\abs{\tilde V_2}-5\eps n_0\right)  (1-\eps)^2n_0p_0$, as claimed.
\end{proof}

With these two lemmas at hand the proof of \autoref{pathExpansion} can be summarized as follows: Use
the weak version implied by \autoref{spannedPaths} to expand to a small fraction of the edges
spanned by $V_{k}, V_{k+1}$. Then invoke \autoref{triangleExpansion} (four times!) to expand to a
$0.52$-fraction of the edges in $V_{k+4}, V_{k+5}$. The details are given below.

\pathExpansion*
\begin{proof}Write $p_0=\alpha p$.
Since $G\in\tildeG{P^2_{k+4}}{n_0}{p_0}{\eps}$ every edge in $G[V_i \cup V_{i+1}]$ forms at least
$(1-\eps)n_0p_0^2$ triangles with $V_{i+2}$. In particular every edge spanned by $V_1\cup V_2$ is
connected to $E(V_k,V_{k+1})$ by $((1-\eps)n_0p^2_0)^{k-1}$ copies of $P^2_{k}$.
Furthermore since $k\geq \frac{3}{\gamma}> \frac{3}{2}\cdot\frac{1+\gamma}{\gamma}$ by
\autoref{spannedPaths} every edge in $E(V_1, V_2)$ is connected to at least 
\[
\frac{((1-\eps)n_0p^2_0)^{k-1}}{2n^{k-3}p^{2k-3}} = \Theta(n^2p) \gg n^2p/n^{\gamma/2}.
\]
distinct edges in $E(V_k, V_{k+1})$ via the square of a path.

Therefore it suffices to show that every set $E_0\subseteq E(V_{k},V_{k+1})$ of size at least
$n_0^2p_0/n^{\gamma/2}$ expands to at least a $0.52$-fraction of $E(V_{k+4}, V_{k+5})$. To this end we
will apply \autoref{triangleExpansion} four times to $G\leftarrow G[V_{k+i}\cup V_{k+i+1} \cup
V_{k+i+2}]$ for $i\in \left\{0,1,2,4\right\}$.

Set $s=\log^2(n)n_0p/n^\gamma$. In $G_{n,p}$ a.a.s.\ all degrees are bounded by $(1+o(1))np$.
Therefore we can find a set of vertices $\tilde V_{k+1}\subseteq V_{k+1}$ of size \[
\frac{\abs{E_0}-n_0s}{\Theta(np)}=\Theta(n^{1-\gamma/2}) \] and an edge set $\tilde E_0\subseteq
E_0$ such that each $v\in \tilde V_{k+1}$ is incident to exactly $s$ edges of $\tilde E_0$.

Apply \autoref{triangleExpansion} with $G\leftarrow G[V_{k}\cup V_{k+1}\cup V_{k+2}]$,  $\tilde
V_2\leftarrow \tilde V_{k+1}$, $\tilde E\leftarrow \tilde E_0$ and denote the set of edges which
form a triangle with some edge of $\tilde E_0$ with $E_1\subseteq E(\tilde V_{k+1}, V_{k+2})$. We
have $\abs{E_1}\geq \frac{p_0}{2p}\abs{\tilde V_{k+1}} n_0 p_0$. Denote with $\tilde
V_{k+2}\subseteq V_{k+2}$ the set of vertices which are incident to more than $s$ edges in $E_1$.
A.a.s.\ in $G_{n,p}$ there exists no set $S$ of size $\abs{\tilde V_{k+1}}=\Theta(n^{1-\gamma/2})$ such
that more than $\sqrt{n}$ vertices have degree at least $2\abs{\tilde V_{k+1}}p$ into $S$. Therefore for $\eps$
small enough depending on $\alpha$
\[
	\abs{\tilde V_{k+2}}
	\geq\frac{\abs{E_1}-\sqrt{n}\cdot\Theta(np)-n_0s}{2\abs{\tilde V_{k+1}}p}
	\geq\frac{\frac{p_0}{3p}\abs{\tilde V_{k+1}} n_0 p_0}{2\abs{\tilde V_{k+1}}p}
	\geq \frac{p_0^2}{6p^2}n_0
	\geq 100\eps n_0.
\]

As before pick a subset $\tilde E_1\subseteq E_1$ such that every $v\in \tilde V_{k+2}$ is incident
to exactly $s$ edges of $\tilde E_1$. Apply \autoref{triangleExpansion} a second time with
$G\leftarrow G[V_{k+1}\cup V_{k+2}\cup V_{k+3}]$, $\tilde V_2\leftarrow \tilde V_{k+2}$, $\tilde
E\leftarrow \tilde E_1$ and denote the set of edges which form a triangle with some edge of $\tilde
E_1$ with $E_2\subseteq E(V_{k+2},V_{k+3})$. We have $\abs{E_2}\geq \frac{p_0}{2p}\abs{\tilde
V_{k+2}}n_0p_0$.  Let $\tilde V_{k+3}\subseteq V_{k+3}$ denote the subset of vertices which are
incident to more than $s'=2\eps n_0p_0$ edges in $E_2$. As $E(V_{k+2}, V_{k+3})$ is \reg{\eps}, by \autoref{l:degreeok}
there are at most $\eps n_0$ vertices in $V_{k+3}$ with more than $(1+\eps)\abs{\tilde V_{k+2}}p_0$
neighbours in $\tilde V_{k+2}$. And by definition of \reg{\eps}ity these vertices are in total incident to at most
$\eps n_0 (1+\eps)|\tilde V_{k+2}| p_0$ edges. Therefore for $\eps$ sufficiently small
\[
	\abs{\tilde V_{k+3}}
	\geq\frac{\abs{E_2}- \eps n_0 (1+\eps)\abs{\tilde V_{k+2}}p_0-n_0s'}{(1+\eps) \abs{\tilde V_{k+2}}p_0}
	\geq \frac{\frac{p_0}{3p}\abs{\tilde V_{k+2}}n_0p_0}{(1+\eps) \abs{\tilde V_{k+2}}p_0}
        \geq  \frac{p_0}{4p}n_0
	\geq 100\eps n_0.
\]

As before pick a subset $\tilde E_2\subseteq E_2$ such that every $v\in \tilde V_{k+3}$ is incident
to exactly $s'$ edges of $\tilde E_2$. Apply \autoref{triangleExpansion} a third time with
$G\leftarrow G[V_{k+2}\cup V_{k+3}\cup V_{k+4}]$, $\tilde V_2\leftarrow \tilde V_{k+3}$, $\tilde
E\leftarrow \tilde E_2$ and denote the set of edges which form a triangle with some edge of $\tilde
E_2$ with $E_3\subseteq E(V_{k+3},V_{k+4})$. We have $\abs{E_3}\geq (1-\eps)^2\left(\abs{\tilde
V_{k+2}}-5\eps n_0\right)n_0p_0\geq \frac{6}{7}\abs{\tilde V_{k+2}} n_0p_0$.  Let $\tilde
V_{k+4}\subseteq V_{k+4}$ denote the subset of vertices which are incident to more than $s'=2\eps
n_0p_0$ edges in $E_3$. As before we obtain
\[
	\abs{\tilde V_{k+4}}
	\geq\frac{\abs{E_3}- \eps n_0 (1+\eps)\abs{\tilde V_{k+3}}p_0-n_0s'}{(1+\eps) \abs{\tilde V_{k+3}}p_0}
	\geq \frac{\frac{5}{7}\abs{\tilde V_{k+3}}n_0p_0}{(1+\eps) \abs{\tilde V_{k+3}}p_0}
        \geq  \frac{4}{7}n_0.
\]
Applying \autoref{triangleExpansion} a fourth time we see $E_3$ expands to $(1-\eps)^2(\abs{\tilde
V_{k+4}}-5\eps n_0)n_0p_0>0.53n^2_0p_0$ edges in $E(V_{k+4}, V_{k+5})$. This concludes the
proof.
\end{proof}

\bibliographystyle{abbrv}
\bibliography{square}

\end{document}